\documentclass[11pt]{article}

\usepackage{tesler}

\title{Tesler matrices and Lusztig data}
\author{Ivan Balashov, Constantine Bulavenko, and Yaroslav Molybog}

\begin{document}
\maketitle

\begin{abstract}
    We study asymptotics of Tesler matrices using Kostant pictures, as well as partial orders on these. We show that the Lusztig data partial order on Kostant pictures refines the `merge' partial order on Kostant pictures, and that the merge partial order on Kostant pictures is equivalent to a partial order on Tesler matrices. This equivalence requires integral flow graphs. Using Kostant pictures we find logarithmic asymptotics of some families of Tesler matrices.
\end{abstract}

\section{Introduction}

This paper is broadly concerned with Kostant's partition function (or KPF for short), which counts the number of ways an element of a semigroup known as the positive root cone (in type $A$) can be decomposed (or partitioned) as a sum of distinguished generators known as positive roots. It is divided into three parts:

In \textit{part 1} (\Cref{s:TeslerMatrices}) we introduce Lusztig data and Kostant pictures (\Cref{ss:L+K}), followed by Tesler matrices and integral flows (\Cref{ss:T+I}). 
Each of these sets has cardinality described by the KPF. 
We draw on their connections to answer the question of asymptotics of the number of regular Tesler matrices, improving the previous benchmark, \Cref{Oneill_Conjecture}.
In the process, we find logarithmic asymptotics for a broad family of generalized Tesler matrices (\Cref{ATheo}).

In \textit{part 2} (\Cref{s:PosetStructures}), we compare Tesler matrices and Lusztig data as posets, again relying on integral flows and Kostant pictures. The partial order on Tesler matrices is due to Armstrong, \cite{Armstrong2014}, while the partial order on Lusztig data is due to Lusztig, \cite{lus90}.
We show that the Tesler poset is equivalent to a natural partial order on Kostant pictures which we call merge (\Cref{thm:IntEqKos}). We also prove (\Cref{refinementTheorem}) that these posets are refined by the Lusztig data poset. 
Lusztig data arises naturally in the representation theory of quantum groups and the theory of crystal bases. In particular, ``multisegments'', which have been shown to model the $B(\infty)$ crystal \cite{claxton2015youngtableauxmultisegmentspbw}, are essentially Kostant pictures rotated 90 degrees.
An important question, which we do not explore here, is whether this inherited crystal structure can be leveraged to study Tesler matrices in their more representation theoretic context --- the diagonal harmonics module.

One goal for the future is to better understand KPF posets. With this in mind, in \textit{part 3} (\Cref{s:ProbabilisticTesler}), we study Markov chains in the Tesler poset. The statistics that arise can give us characteristic values for each poset structure. In addition, we consider ideas to partially generate the poset in order to reduce computing time, while still estimating certain values, for example the  M\"obius function, which can be used to study asymptotics \cite{ONeill}.

\section{Tesler matrices, Lusztig data, and their asymptotics}
\label{s:TeslerMatrices}
\subsection{Lusztig data and Kostant's partition function}
\label{ss:L+K}
First of all, we shall define roots and Lusztig data.
\begin{definition}
    Let $\left(e_1,\,e_2,\,...,\, e_{n+1}\right)$ be the standard vector basis of $\R^{n+1}$. A sum $\sum c_{ij}\alpha_{ij} = \varnu$, where $\alpha_{ij}=e_i-e_{j+1},\,1\leq i\leq j\leq n$ is called a \new{weight}. We will sometimes abbreviate $\alpha_{ii}=\alpha_i$. We consider weights $\varnu$ with $c_{ij}\in\N$. 
    A \new{Lustzig datum} for $\varnu$
    is any tuple
    $$
        \mathbf{a}=\left(a_{11},\,a_{12},\,...,\,a_{1n},\,a_{22},\,a_{23},\,...,\,a_{2n},\,...,\,a_{nn}\right) \in \N^{\binom{n+1}{2}}
    $$
    such that $\sum a_{ij} \alpha_{ij}$ is again $\varnu$. This order on coefficients is the standard order; it corresponds to the standard reduced word for the longest element of $S_{n+1}$, the symmetric group on $n+1$ elements. 
    The set of all possible Lusztig data for a given $\varnu$ is denoted as $A\left(\varnu\right)$.
\end{definition}
A more intuitive representation of a decomposition of a weight $\varnu$ due to \cite{AnKo} is as follows.
\begin{definition}
    A \new{Kostant picture} for $\varnu=\sum c_{ij}\alpha_{ij}$ is a diagram on $n$ ordered vertices. In it, each term $\alpha_{ij}$ is represented by a loop around (or a bar above) vertices $i$ through $j$. The set of all possible Kostant pictures for a given $\varnu$ is denoted by $\K\left(\varnu\right)$, while their number $\KPF\left(\varnu\right)$ defines the \new{Kostant partition function}.
\end{definition}
\begin{example}
    For clarity, we shall tabulate every Lusztig data and its corresponding Kostant picture for $\varnu=\left(1,\,1,\,-1,\,-1\right)$ in \Cref{iLTable}.
\begin{table}[h]
\centering
\begin{tabular}{ >{\centering\arraybackslash}p{0.22\textwidth}  >{\centering\arraybackslash}p{0.22\textwidth} >{\centering\arraybackslash}p{0.22\textwidth}}
\hline
 Lusztig data & Kostant picture $\varnu$ & Decomposition \\ 
\hline & & \\[-1em] 
$\left(1,\,0,\,0,\,2,\,0,\,1\right)$&$\dash{*}\:\dash{\dash{*}}\:\dash{*}$ & $(\alpha_1) + 2 (\alpha_2) + (\alpha_3)$ \\
$\left(0,\,1,\,0,\,1,\,0,\,1\right)$&$\dash{*\:}\dash{\dash{*}}\:\dash{*}$ & $(\alpha_{12}) + (\alpha_2) + (\alpha_3)$ \\
$\left(1,\,0,\,0,\,1,\,1,\,0\right)$&$\dash{*}\:\dash{\dash{*}}\dash{\:*}$ & $(\alpha_1) + (\alpha_2) + (\alpha_{23})$ \\
$\left(0,\,1,\,0,\,0,\,1,\,0\right)$&$\dash{*\:}\dash{\dash{*}\:*}$ & $(\alpha_{12}) + (\alpha_{23})$ \\
$\left(0,\,0,\,1,\,1,\,0,\,0\right)$&$\dash{*\:}\dash{\dash{*}}\dash{\:*}$ & $(\alpha_{13}) + (\alpha_2)$ \\
  \hline
\end{tabular}
\caption{$A\left(\varnu\right)$ and $\K\left(\varnu\right)$ for $\varnu = \left(1,\,1,\,-1,\,-1\right)$}
\label{iLTable}
\end{table}

It will be convenient for us to describe Kostant pictures not by their weight $\varnu$, but their \new{height} $\eta = (\eta_1,\,\dots,\,\eta_n)$ where $ \eta_i$ are such that $\varnu = \sum\eta_i \alpha_i$. 
Note that if $\varnu = (\varnu_1,\dots,\varnu_{n+1})\in\R^{n+1}$ then $\eta=\left(\varnu_1,\,\varnu_1+\varnu_2,\,...,\,\sum_{k=1}^{n}{\varnu_k}\right)$. We denote the set of Kostant pictures of height $\eta$ by $K[\eta]$ putting height in square brackets. For example, we have $\K\left(1,\,1,\,-1,\,-1\right)=\K\left[1,\,2,\,1\right]$.
\end{example}
\subsection{Tesler matrices and integral flows}
\label{ss:T+I}
Now, we shall define Tesler matrices. Let $U_n$ be the set of $n\times n$ upper-triangular matrices with non-negative integer entries.
\begin{definition}
    For a matrix $A = \left(a_{ij}\right) \in U_n$ we define its $k$\textsuperscript{th} \new{hook sum} $h_k,\,1\leq k\leq n$ as
    $$
        h_k=\sum_{i=k}^n{a_{ki}}-\sum_{i=1}^{k-1}{a_{ik}}
    $$ 
    and its \new{hook sum vector} as $\mathbf{ h}=\left(h_1,\,h_2,\,...,\,h_n\right)$.
\end{definition}
\begin{definition}
    The set of all $U_n$ matrices with given hook sum vector $\mathbf{h}$ is called the set of \new{(generalized) Tesler matrices with hook sum $\mathbf h$} and is denotes as $\T\left(\mathbf{h}\right)$.
\end{definition}
\begin{example}\label{Tex}
    When the hook sum is $\left(1,\,1,\,...,\,1\right)=\mathbf{1}^n$, then $\T\left(\mathbf{1}^n\right)$ is called the set of \new{regular Tesler matrices}. For $n=2$ we have only two regular Tesler matrices:
\begin{align*}
\begin{bmatrix}
1 & 0 \\
& 1 
\end{bmatrix}\;\;\textsc{and}\;\;
\begin{bmatrix}
0 & 1 \\
& 2 
\end{bmatrix}.
\end{align*}
\end{example}
At first glance, Tesler matrices have little to do with Lusztig data or Kostant pictures. There is a representation that links both, though.
\begin{definition}
    For a vector $\mathbf{h}=\left(h_1,\,h_2,\,...,\,h_n\right)$ with integer entries, we let $\mathcal I\left(\mathbf{h}\right)$ denote the set of the \new{integral flow graphs} on $n+1$ vertices with the net flow $\left(\bh,\,-\sum {h_k}\right)$ such that every flow on the edges is non-negative (\Cref{fig:exintflow}).
\begin{figure}[h]
\centering
\begin{tikzpicture}
    \node (v1) at (0,0) [circle, fill, radius = 0.1cm, label = below:$h_1$] {};
    \node (v2) at (2,0) [circle, fill, radius = 0.1cm, label = below:$h_2$] {};
    \node (v3) at (4,0) [circle, fill, radius = 0.1cm, label = below:$h_3$] {};
    \node (v4) at (6,0) [circle, fill, radius = 0.1cm, label = below:$-h_1-h_2-h_3$] {};
    \draw[->] (v1) -- (v2); 
    \draw[->] (v2) -- (v3);
    \draw[->] (v3) -- (v4);
    \draw[->] (v1) to [out=30,in=150] (v3);
    \draw[->] (v2) to [out=30,in=150] (v4); 
    \draw[->] (v1) to [out=45,in=135] (v4); 
\end{tikzpicture}
\caption{Integral flow graph for $n=3$}
\label{fig:exintflow}
\end{figure}
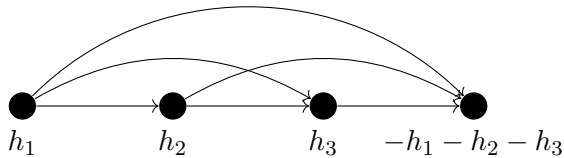
\end{definition}
It has been shown \cite{M_sz_ros_2016} that $\left|\T\left(\mathbf{h}\right)\right|=\left|\mathcal{I}\left(\mathbf{h}\right)\right|=\KPF\left(\mathbf{h},\,-\sum{h_i}\right)$.
This can be illustrated with a bijection between the three structures.

Firstly, given a Tesler matrix $A = \left(a_{ij}\right)$, we construct an integral flow with an edge from the $i$\textsuperscript{th} vertex to the $j$\textsuperscript{th} vertex assigned the value $a_{ij}$ for each $1\leq i<j\leq n$, while an edge from the $i$\textsuperscript{th} vertex to the last vertex is assigned the value $a_{ii}$ (\Cref{fig:bijectionTesIntex}). 
Secondly, we draw a Kostant picture with the number of loops around vertices $i$ through $j$ being the flow value from the $i$\textsuperscript{th} vertex to the $j$\textsuperscript{th} vertex.
\begin{figure}[h]
    \centering
    \includegraphics[height = 0.12\textheight]{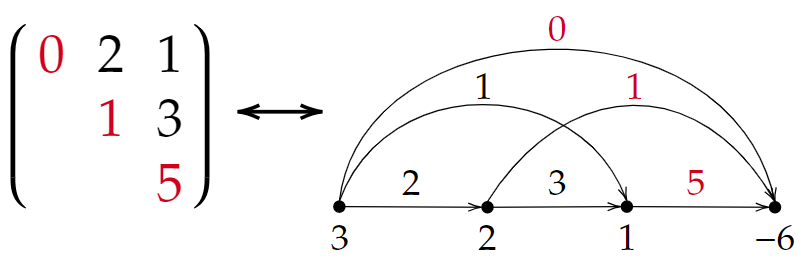}
    \caption{A Tesler matrix and an integral flow correspondence, $\mathbf{h}=\left(3,\,2,\,1\right)$}
    \label{fig:bijectionTesIntex}
\end{figure}

\subsection{Asymptotics}
\label{ss:Asymptotics}

The question of cardinality asymptotics was raised in \cite{ONeill} in the context of Tesler matrices. Thus, we shall consider sequences of type $\left|\T\left(\mathbf{h}\left(n\right)\right)\right|$ where the hook length depends on $n$. Since such sequences tend to grow exponentially, it is preferred to use $L\left(\mathbf{h}\right)=\ln \left|\T\left(\mathbf{h}\right)\right|$.

For example, it is easy to verify that $L\left(1,\,\mathbf{0}^n\right)=n\ln2$. 
More importantly for us, it was shown (by a clever appeal to the Morris identity) in \cite{zeilberger1998proof} that $$L\left(1,2,\dots,n\right)=\sum_{i=1}^n{\ln C_i}\sim n^2\ln2,$$ where $C_i$ denotes $i$\textsuperscript{th} Catalan number $\frac{1}{1+i}\binom{2i}{i}$ and the asymptotics notation is from \Cref{apx:Asymptotics}.

Regarding the regular Tesler matrices $\mathcal T\left(\mathbf{1}^n\right)$, \Cref{Oneill_Conjecture} was posited by Jason O'Neill. If true, it would imply directly the limit $L\left(\mathbf{1}^n\right)\asymp n^2$.
\begin{conjecture}[\cite{ONeill}]\label{Oneill_Conjecture}
    Let $A\in\T\left(\mathbf{1}^n\right)$ and let $\mu$ be the M\"{o}bius function for the corresponding Tesler poset (\Cref{def:PosetTesler}). Then $$\left|\mu\left(\hat{0},\,A\right)\right|\leq n!$$
\end{conjecture}
To exploit the connection between Tesler matrices and the KPF, we introduce a visually intuitive representation for a weight $\varnu$ called a \new{height diagram}. The height diagram of a Kostant picture for $\varnu = \sum \eta\left(i\right)\alpha_i$ is simply the graph of $y=\eta\left(x\right)$ in the plane. It is a discrete function that is zero on its entire domain except for a finite number of arguments. That being said, we shall draw in continuously for the sake of clarity.

The utility of height diagrams lies in the following easy observation. If a height diagram for $\varnu = \sum\eta(i)\alpha$ covers a height diagram for $\varnu'=\sum\eta'(i)\alpha_i$ in the sense that there exists $a\in\N$ such that $\eta(i-a) \ge \eta(i)$ for all $i\in\R$, then $\KPF(\varnu) \ge \KPF(\varnu')$. Thus, we can approximate complex height diagrams with simpler ones, such as $\eta=\mathbf{n}^m$ the height diagram of $\mathbf{h}=\left(n,\, \mathbf{0}^{m-1}\right)$. Let us introduce the notation $\T\left[n,\,m\right]=\T\left(n,\, \mathbf{0}^{m-1}\right)$. The next proof demonstrates an application. 
\begin{proposition}\label{BasicHeight}
    $L\left(1,\,2,\,...,\,n\right)\asymp L\left[n^2,\,n\right]$ and $L\left(\mathbf{1}^n\right)\asymp L\left[n,\,n\right]$.
\end{proposition}
\begin{proof}
    For $L\left(1,\,2,\,...,\,n\right)$ we have $\eta\left(x\right)=\frac{1}{2}\left(x^2+x\right)$, while for $L\left[\frac{1}{2}\left(n^2+n\right),\,n\right]$ it is $\eta\left(x\right)=\frac{1}{2}\left(n^2+n\right)$. Therefore (\Cref{examplehight1}), $L\left(1,\,2,\,...,\,n\right)\leq L\left[\frac{1}{2}\left(n^2+n\right),\,n\right]\leq L\left(1,\,2,\,...,\,2n\right)$, implying the first equation since  $L\left[\frac{1}{2}\left(n^2+n\right),\,n\right]\asymp L\left[n^2,\,n\right]$. The same can be stated for the second one (\Cref{examplehight2}).
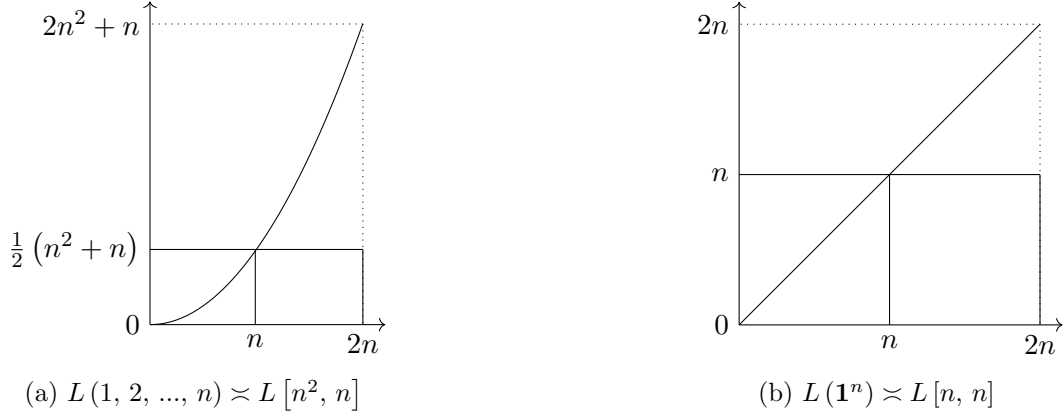
\begin{figure}[h]
\centering
\begin{subfigure}[b]{0.45\textwidth}
\centering
\begin{tikzpicture}
\draw[->] (0, 0) -- (3.13, 0);
\draw[->] (0, 0) -- (0,4.3);
\node[left] at (0, 0)  {$0$};
\draw[domain=0:2.83, smooth, variable=\x] plot ({\x}, {\x*\x/2});
\draw[-]  (2.83, 1)--(0, 1) node[left] {$\frac{1}{2}\left(n^2+n\right)$};
\draw[-] (1.4, 1)--(1.4, 0) node[below] {$n$};
\draw[-] (2.83, 1)--(2.83, 0) node[below] {$2n$};
\draw[dotted] (2.83, 4)--(2.83, 0);
\draw[dotted] (2.83, 4)--(0, 4) node[left] {$2n^2+n$};
\end{tikzpicture}
\caption{$L\left(1,\,2,\,...,\, n\right)\asymp L\left[n^2,\,n\right]$}
\label{examplehight1}
\end{subfigure}
\hfill
\begin{subfigure}[b]{0.45\textwidth}
\centering
\begin{tikzpicture}
\draw[->] (0, 0) -- (4.3, 0);
\draw[->] (0, 0) -- (0,4.3);
\draw[-] (4,4)--(0, 0) node[left] {$0$};
\draw[-]  (4, 2)--(0,2) node[left] {$n$};
\draw[-] (2, 2)--(2, 0) node[below] {$n$};
\draw[-] (4, 2)--(4, 0) node[below] {$2n$};
\draw[dotted] (4, 4)--(4, 0);
\draw[dotted] (4, 4)--(0, 4) node[left] {$2n$};
\end{tikzpicture}
\caption{$L\left(\mathbf{1}^n\right)\asymp L\left[n,\, n\right]$}
\label{examplehight2}
\end{subfigure}
\caption{Basic examples of height diagrams}
\end{figure}
\end{proof}
Now, we are ready to prove an important statement everything else will be based on.
\begin{proposition}\label{DecProp}
    Let $h,\,w,\,w_0:\Z^+\to\Z^+$ such that $w_0\succ w\succ 1$. Then 
    $$\max\left(\frac{w_0}{w}L\left[h,\,w\right],\,L^{\frac{1}{4}w^2}\left[h,\,\frac{w_0}{w}\right]\right)\precsim L\left[h,\, w_0\right]\precsim \frac{w_0}{w}L\left[h,\,w\right]+ L^{w^2}\left[h,\,\frac{w_0}{w}\right],$$
    where $T^k\left[\eta\right]$ denotes the maximal number of colorings a Kostant picture with height $\eta$ can have if every loop is colored by one of $k$ distinct colors.
\end{proposition}
\begin{proof}
   We can decompose $\eta=\left[h,\,w_0\right]$ into $\frac{w_0}{w}$ independent $\left[h,\,w\right]$ (\Cref{HDecompose}). If we ignore all possible loops $\alpha_{ij}$ that are contained within several such divisions, we have $L\left[h,\,w_0\right]\succsim \frac{w_0}{w} L\left[h,\, w\right]$.

   Now, consider all such missing $\alpha_{ij}$ as determined by two parameters: the set of divisions the loop covers and the relative positions of its ends in the first and the last divisions. Note that in order to produce a lower bound, this situation can be modeled by $\K\left[h,\,\frac{w_0}{w}\right]$ with each loop having one of $\left(\frac{w}{2}\right)^2$ possible colors ($\frac{w}{2}$ possible choices of each loop end so that the loops with disjoint division sets do not cross). Moreover, since $L\left[h,\,\frac{w_0}{w}\right]\prec w L\left[h,\,\frac{w_0}{w}\right]\precsim  L\left[h,\,w_0\right]$, we are only concerned with the maximal number of colorings a picture can have.
   If we allows all $w$ possible choices per end, we achieve the upper-bound.
\begin{figure}[h]
\centering
\begin{tikzpicture}
\draw[->] (0, 0) -- (12.3, 0);
\draw[->] (0, 0) -- (0,4.3);
\node[left] (0, 0)  {$0$};
\draw[-] (12,4)--(0, 4) node[left] {$h$};
\draw[-] (12,4)--(12, 0) node[below] {$w_0$};
\draw[-] (2,4)--(2, 0) node[below] {$w$};
\draw[-] (4,4)--(4, 0) node[below] {$2w$};
\draw[-] (6,4)--(6, 0) node[below] {$3w$};
\node at (9, 2) {$\cdots$};
\draw[thick,  double] (1.2,0.3)--(11.5, 0.3);
\draw[thick, dashed] (3.4,.5)--(6.5, .5);
\draw[thick, densely dotted] (0.3,.5)--(2.7, .5);
\draw[ultra thin] (2.6,.7)--(6.2, .7);
\end{tikzpicture}
\caption{Decomposition of $T\left[h,\,w_0\right]$}
\label{HDecompose}
\end{figure}
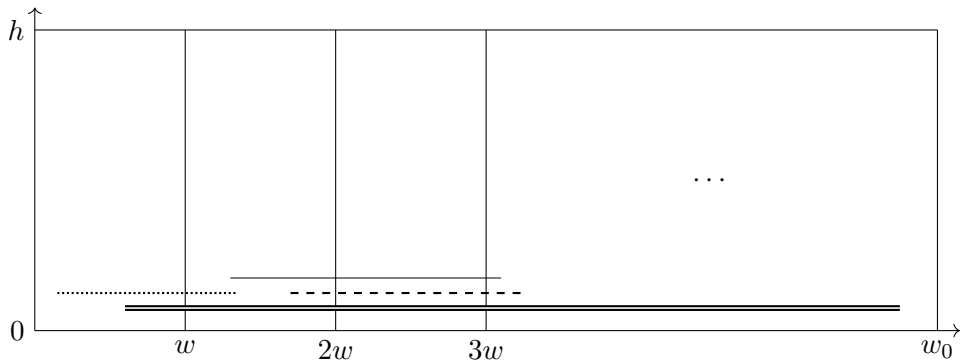
\end{proof}
\begin{lemma}\label{L^k}
    Let $h,\,w,\,k:\Z^+\to\Z^+$ such that $k\succ 1$. Then
    $$L^k\left[h,\, w\right]\asymp
\begin{dcases}
    hw\ln\left(\frac{k}{h}\right)&\text{if } h\prec k\\
    w\sqrt{hk}&\text{if } k\preccurlyeq h\preccurlyeq kw^2\\
    kw^2\ln\left(\frac{h}{kw^2}\right)&\text{if } h\succ kw^2
\end{dcases}$$
\end{lemma}
\begin{proof}
    We shall consider every Kostant picture $\sum{c_{ij}\alpha_{ij}}$ with $\eta=\left[h,\, w\right]$ as matrix $C=\left(c_{ij}\right)\in U_n$. Note that every such matrix is determined by $$s_n=\sum_{i\leq n\leq j}{c_{ij}}=\eta_n.$$ Crucially for us, we can now approximate $L^k\left[h,\,w\right]$: $$L^k\left[h,\, w\right]=\max{\sum_{c=c_{ij}}{\ln\binom{c+k-1}{k-1}}}\sim\max{\sum_{c=c_{ij}\neq 0}{\left(c\ln\left(1+\frac{k}{c}\right)+k\ln \left(1+\frac{c}{k}\right)\right)}}.$$

    Note that we are asymptotically (in terms of $\asymp$) allowed not to consider the entries $c_{ij}$ with either $i\geq w-\Delta w +1$, $j\leq \Delta w$ , or $j-i\geq \Delta w= \frac{1}{k}\times w$ for a constant $k=3$, for example. This leaves us with a trapezoid matrix, which we can divide into diagonals indexed $1$ through $\Delta w$ (\Cref{TMatrix}). Consider $\sum{s_n}\asymp hw$. Every entry on $n$\textsuperscript{th} diagonal goes into this sum with the coefficient $n$, meaning that we can fill the maximum of $D\asymp \min\left(w,\,\sqrt{h}\right)$ first diagonals by non-zero entries.
\begin{figure}[h]
    \centering
    \begin{tikzpicture}
    \node at (0,0) {$\begin{bmatrix}
*&*&*&*&*&*&*&*&*\\
&*&*&3&*&*&*&*&*\\
&&*&2&3&*&*&*&*\\
&&&1&2&3&*&*&*\\
&&&&1&2&3&*&*\\
&&&&&1&2&3&*\\
&&&&&&*&*&*\\
&&&&&&&*&*\\
&&&&&&&&*\\
\end{bmatrix}$};

\draw (-0.69,0.282)
  -- (1.06,0.282)
  -- (-.69,1.841)
  -- cycle;
  \draw (-.15,-.2)
  -- (1.6,-.2)
  -- (-.15,1.36)
  -- cycle;
  \draw (.39,-0.682)
  -- (2.14,-0.682)
  -- (.39,0.879)
  -- cycle;

\end{tikzpicture}
    \caption{Three diagonals of a trapezoid matrix}\label{TMatrix}
\end{figure}

Now, consider the matrix which maximizes the sum we are interested in. Bring every non-zero entry to a vacant place on a lower diagonal if there are any left, and replace every entry with the arithmetic average on its diagonal. Not only does this have no impact on the asymptotics, it also preserves $s_n\leq h$. Let $c_n$ be the value all entries we now have on $n$\textsuperscript{th} diagonal, then $$L^k\left[h,\, w\right]\asymp w\sum_{n=1}^{D}{c_n\ln\left(1+\frac{k}{c_n}\right)}+kw\sum_{n=1}^{D}{\ln \left(1+\frac{c_n}{k}\right)}.$$

It is easy to check that the right-hand side is maximized iff $\ln\left(1+\frac{k}{c_n}\right)=n\ln\left(1+\frac{k}{c_1}\right)$. Let us denote $a=1+\frac{k}{c_1}$, then $c_n=\frac{k}{a^n-1}$, allowing us to rewrite the sum as 
$$L^k\left[h,\, w\right]\asymp hw\ln a+ kw\sum_{n=1}^{\infty}{\frac{1}{n}\times \frac{1-a^{-nD}}{a^n-1}},\, \sum_{n=1}^{D}{\frac{n}{a^n-1}}=\frac{h}{k}.$$

What is left is to work through all possible variants.
\begin{enumerate}[label=(\roman*)]
     \item If $h\prec k$, then $$\frac{h}{k}=\sum_{n=1}^{D}{\frac{n}{a^n-1}}\sim \frac{1}{a}.$$
     Substituting this, we get $L^k\left[h,\, w\right]\asymp hw\ln\left(\frac{k}{h}\right)$.
     \item If $h\asymp k$, then by similar analysis $a-1\asymp 1$, leading to $L^k\left[h,\, w\right]\asymp hw$.
     \item If $k\prec h\preccurlyeq kw^2$, then we shall impose the following bounds: $$\int_1^{D+1}{\frac{x\, dx}{a^x-1}}\leq\sum_{n=1}^{D}{\frac{n}{a^n-1}}\leq \int_0^D{\frac{x\, dx}{a^x-1}}.$$ This leads to $\ln a\asymp \sqrt{\frac{k}{h}}$ and $L^k\left[h,\, w\right]\asymp w\sqrt{hk}$. 
     \item If $h\succ kw^2$, then $D\asymp w$ and $$\frac{h}{k}=\sum_{n=1}^{D}{\frac{n}{a^n-1}}\sim \frac{D}{\ln a}.$$ This leads to $L^k\left[h,\, w\right]\asymp kw^2\ln\left(\frac{h}{kw^2}\right)$.
\end{enumerate}
\end{proof}
\begin{theorem}\label{ATheo}
Let $h,\,w:\Z^+\to\Z^+$ such that $h,\, w\succ 1$. If $w\prec \sqrt{h}$, then $L\left[h,\, w\right]\sim w^2\ln\left(\frac{\sqrt{h}}{w}\right)$. If $w\succ \sqrt{h}$, then $L\left[h,\, w\right]\sim \gamma_+ w\sqrt{h}$, where $\gamma_+$ is universal constant.
\end{theorem}
\begin{proof}
    Using the first two cases of \Cref{L^k}, we can deduce $L\left[n^2,\, w\right]\sim \gamma_+ wn,\,w\succ n$ for some $\gamma_+$ (\Cref{DecProp}) from $L\left[n^2,\,n\right]\asymp n^2$. From this follows $L\left[h,\, w\right]\sim \gamma_+ w\sqrt{h},\, w\succ \sqrt{h}\succ 1$.

    Consider $w\succ\sqrt{h}\succ 1$. The third case of \Cref{L^k} leads to $w^2\ln\left(\frac{\sqrt{h}}{w}\right)\preccurlyeq L\left[h,\,w\right]\preccurlyeq w\sqrt{h}$. Let $T_2\left[h,\,w\right]$ denote the maximal number of ways to split a Kostant picture with height $\left[h,\,w\right]$ into two Kostant pictures with heights $\left[\frac{h}{2},\,w\right]$. Assume $L_2\left[h,\,w\right]\prec L\left[h,\,w\right]$, then $L\left[2h,\,2w\right]\succsim 4L\left[h,\,w\right]$, which is impossible. Thus, $L\left[h,\,w\right]\asymp L_2\left[h,\,w\right]\precsim L^2\left[h,\,w\right] $. Note that we can approximate $L^2\left[h,\,w\right]$ much like we did in \Cref{L^k}. In fact, the same manipulations with the trapezoid matrix lead to 
    $$
        L^2\left[h,\, w\right]\asymp w^2\ln\left(\frac{c_1}{w!}\right),\, c_1\asymp\frac{h}{w}.
    $$
    So we have $w^2\ln\left(\frac{\sqrt{h}}{w}\right)\preccurlyeq L\left[h,\,w\right]\preccurlyeq  L^2\left[h,\,w\right]\asymp w^2\ln\left(\frac{\sqrt{h}}{w}\right)$. 

    Now, let $L\left[h,\,w\right]\precsim \gamma_- w^2\ln\left(\frac{\sqrt{h}}{w}\right)$ for some height $\eta=\left[h,\,w\right]$. Denote the number of Kostant pictures with this height such that for every $\alpha_{ij}$ either $i=j$ or $j-i\geq k$ holds as $T\left[h,\,w\right]_{\geq k}$. Similarly defining $T\left[h,\,w\right]_{\leq k}$, we have $$L\left[h,\,w\right]\leq  L\left[h,\,w\right]_{\geq w/c}+L\left[h,\,w\right]_{\leq w/c}\precsim w^2\ln\left(\frac{\sqrt{h}}{w}\right)+\frac{3\gamma_-}{c}w^2\ln\left(\frac{\sqrt{h}}{w}\right),$$ for every positive integer $c$. Since $L\left[h,\,w\right]\succsim w^2\ln\left(\frac{\sqrt{h}}{w}\right)$, we have $L\left[h,\,w\right]\sim w^2\ln\left(\frac{\sqrt{h}}{w}\right)$.
\end{proof}
\begin{corollary}
\label{CorSum}
    Let $\eta:\Z^+\to\Z^+$ be a non-decreasing function such that $n\succ\sqrt{\eta\left(n\right)}\succ 1$ and $\sum{\sqrt{\eta\left(i\right)}}\succ \eta\left(n\right)$. Then $$L\left[\eta\right]_n=L\left[\eta\left(1\right),\,\eta\left(2\right),\,...,\,\eta\left(n\right)\right]\sim \gamma_+ \sum_{i=1}^n{\sqrt{\eta\left(i\right)}}.$$
\end{corollary}
\begin{proof}
    For $n\succ w\succ \sqrt{\eta\left(n\right)}$ we have 
    $$ 
        \gamma_+ w\sum_{i=1}^{\frac{n-w}{w}}{\sqrt{\eta\left(iw\right)}}\precsim L\left[\eta\right]_n\precsim \gamma_+ w\sum_{i=1}^{\frac{n}{w}}{\sqrt{\eta\left(iw\right)}}.
    $$
    We shall choose such $w$ that $\sum{\sqrt{\eta\left(i\right)}}\succ  w\sqrt{\eta\left(n\right)}$, then we get the proposed asymptotics. 
\end{proof}
\begin{corollary}
\label{ATesler}
    $L\left(\mathbf{1}^n\right)\sim \frac{2}{3}\gamma_+ n\sqrt{n}.$
\end{corollary}

\section{Partial order on KPF-sets}
\label{s:PosetStructures}

\subsection{Merge order}
\label{s:TeslerPoset}
\begin{definition}[\cite{ONeill}]\label{def:PosetTesler}
    To define a \new{poset on Tesler matrices}, first fix a hook sum vector $\bh$ and define a covering relation $\left(b_{ij}\right) \lessdot \left(a_{ij}\right)$ in $\T\left(\bh\right)$ iff they have the same entries except 
    $$
    \begin{gathered}
        a_{ij}=b_{ij}+1,\,a_{jk}=b_{jk}+1,\,a_{ik}=b_{ik}-1 \text{ for a unique triple } i<j<k\\ 
        \text{or }a_{ij}=b_{ij}+1,\,a_{jj}=b_{jj}+1,\,a_{ii}=b_{ii}-1 \text{ for a unique pair }i<j.
    \end{gathered}
    $$
\end{definition}

An example can be seen in \Cref{fig:tesposex}. Since $\cI(\bh)$ is in bijection with $\T(\bh)$, we can transport this partial order to the integer flow setting where it becomes the order defined by the covering relation seen in \Cref{fig:corintpos}. We refer to this partial order (in either setting) as the \new{merge} order.
\begin{figure}[h]
\begin{subfigure}{.5\textwidth}
  \centering
  \includegraphics[width=.6\linewidth]{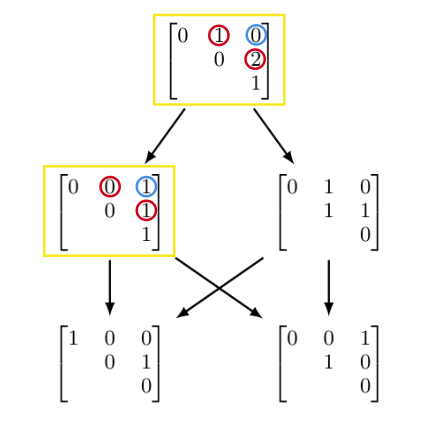}  
  \caption{The $\T\left(1,\,1,\,-1\right)$ poset}
  \label{fig:tesposex}
\end{subfigure}
\begin{subfigure}{.5\textwidth}
  \centering
  \includegraphics[width=.7\linewidth]{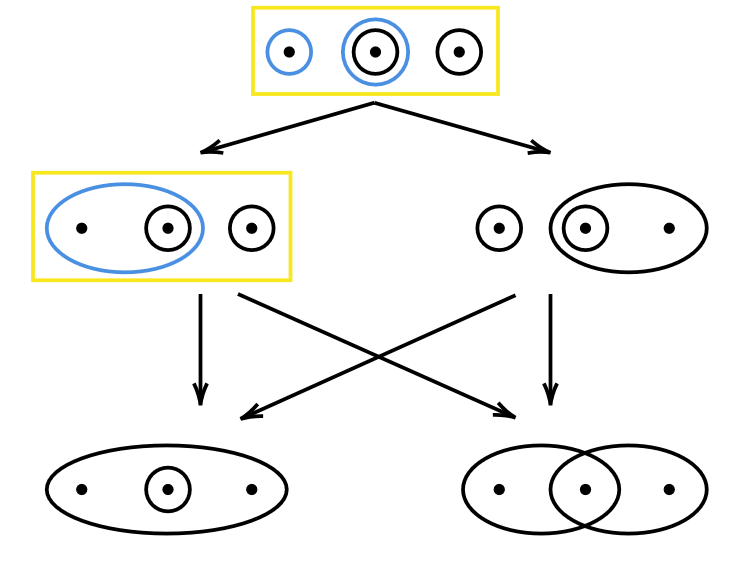}  
  \caption{The $\K\left(1,\,1,\,-1,\,-1\right)$ poset}
  \label{fig:mergposkost}
\end{subfigure}
\caption{Merge order on Tesler matrices and Kostant pictures}
\label{fig:fig}
\end{figure}
\begin{figure}[h]
    \centering
    \includegraphics[height = 0.25\textheight]{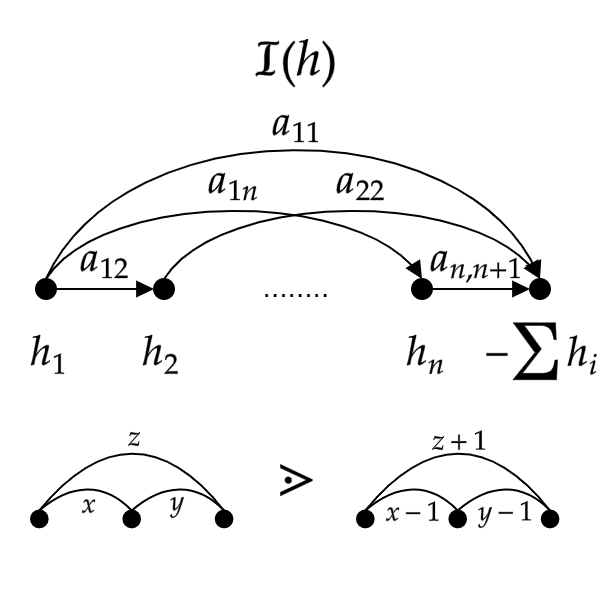}
    \caption{The corresponding integral flow poset}
    \label{fig:corintpos}
\end{figure}

Now we shall define the merge order on Kostant pictures.
\begin{definition}
\label{d:mer-ord}
    We define the \new{poset on Kostant pictures} by letting one picture be covered by another iff it can be acquire from the second one by merging two of its loops into one (\Cref{fig:mergposkost}). 
\end{definition}

We consider these partial orders because they are ranked, which makes them comfortable to work with. For example, the rank functions, mobius functions, and etc. can all be calculated.

\begin{lemma}
\label{thm:IntEqKos}
        The merge order on $\cI(\bh)$ and $\K\left(\mathbf{h},\,-\sum{h_k}\right)$ are isomorphic.
\end{lemma}
\begin{proof}
    By construction there is already a bijection. However, it is not obvious whether the partial order will hold.

    Let $i$, $k$, and $j$ be any natural numbers such that $1 \leq i < k < j \leq n$. Then consider the diagram for the integral flow with its partial order on $\cI(\bh)$. Now we construct the corresponding diagrams as Kostant pictures As we can see, the two extra loops disappear and create a larger one. This is equivalent to the partial order on $\K\left(\mathbf{h},\,-\sum{h_k}\right)$ (\Cref{MOF}).
    
    Note that it is important to check the other way around (some other posets might not work), however, in this case it's analogous. Also, consider the case $j=n+1$ as the same except in the diagrams $a_{ij}$ is $a_{ii}$ and $a_{kj}$ is $a_{kk}$.
\begin{figure}[h]
\centering
\begin{subfigure}{.6\textwidth}
  \centering
  \includegraphics[height = 0.09\textheight]{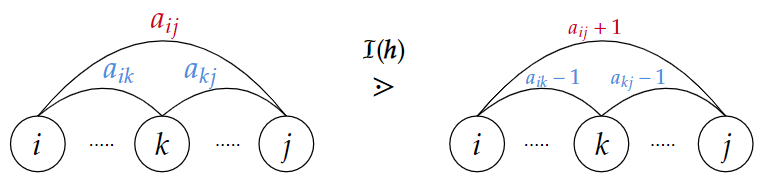}
\end{subfigure}
\hfill
\begin{subfigure}{.6\textwidth}
  \centering
  \includegraphics[height = 0.09\textheight]{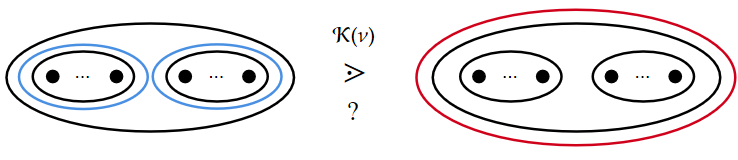}
\end{subfigure}
\caption{The merge flow covering relation correspondance}
\label{MOF}
\end{figure}
\end{proof}

\subsection{Two-sided dictionary order, how is it different?}
\label{ss:LusztigPoset}

\begin{definition}[\cite{Tingley}] 
\label{d:lus-ord}
    In order to introduce the \new{poset on Lusztig data}, the \new{two-sided dictionary} partial order on $A\left(\varnu\right)$ 
    is defined by $\mathbf{a}\le\mathbf{a'}$
    if there can be found two integers $l\leq r$ such that 
    \begin{itemize}[label = $\bullet$]
        \item $a'_l>a_l$, 
        \item $a'_r>a_r$,
        \item $a'_i=a_i$ for all $i<l$ and $i > r$.
    \end{itemize} 
\end{definition}

\begin{figure}[h]
    \centering
\begin{tikzpicture}[>=latex,line join=bevel,]
\node (node_0) at (75.0bp,8.5bp) [draw,draw=none] {$\left(0, 0, 1, 1, 0, 0\right)$};
  \node (node_1) at (75.0bp,61.5bp) [draw,draw=none] {$\left(0, 1, 0, 0, 1, 0\right)$};
  \node (node_2) at (33.0bp,114.5bp) [draw,draw=none] {$\left(\textcolor{red}{0}, 1, 0, \textcolor{red}{1}, 0, 1\right)$};
  \node (node_3) at (117.0bp,114.5bp) [draw,draw=none] {$\left(1, 0, 0, 1, 1, 0\right)$};
  \node (node_4) at (75.0bp,167.5bp) [draw,draw=none] {$\left(\textcolor{red}{1}, 0, 0, \textcolor{red}{2}, 0, 1\right)$};
  \draw [black,->] (node_0) ..controls (75.0bp,23.805bp) and (75.0bp,34.034bp)  .. (node_1);
  \draw [black,->] (node_1) ..controls (62.723bp,77.408bp) and (53.276bp,88.88bp)  .. (node_2);
  \draw [black,->] (node_1) ..controls (87.277bp,77.408bp) and (96.724bp,88.88bp)  .. (node_3);
  \draw [black,->] (node_2) ..controls (45.277bp,130.41bp) and (54.724bp,141.88bp)  .. (node_4);
  \draw [black,->] (node_3) ..controls (104.72bp,130.41bp) and (95.276bp,141.88bp)  .. (node_4);
\end{tikzpicture} 
\caption{Poset on $A\left(1,\,1,\,-1,\,-1\right)$}
\end{figure}
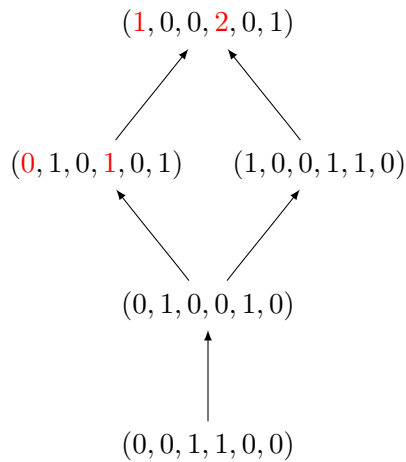

\begin{conjecture}[\cite{NathanWilliams2022}]
    The partial order from \Cref{d:lus-ord} and the partial order from \Cref{d:mer-ord} are equivalent.
\end{conjecture}
\begin{theorem}
\label{refinementTheorem}
    The poset coming from Lusztig data is a \new{refinement} (the cardinalities are equal, however, the edges are preserved only in one direction) of the poset from $\cK(\varnu)$.
\end{theorem}
\begin{proof}
   Let's consider the merge order on Kostant pictures and read off its Lusztig data. We get that $(...,\alpha_{ij},...,\alpha_{ik},...,\alpha_{j+1k},...) \succeq (...,\alpha_{ij}-1,...,\alpha_{ik}+1,...,\alpha_{j+1k}-1,...)$ under $\cK(\varnu)$. Additionally, the partial order holds under $A(\varnu)$ (\Cref{fig:corluszdat}). That means every covering in the Kostant pictures is a covering in the Lusztig data. However, the reverse isn't true. An example can be shown for $\varnu = (1,1,-1,-1)$ (\Cref{fig:counterexample}).
\begin{figure}[h]
    \centering
    \includegraphics[height = 0.175\textheight]{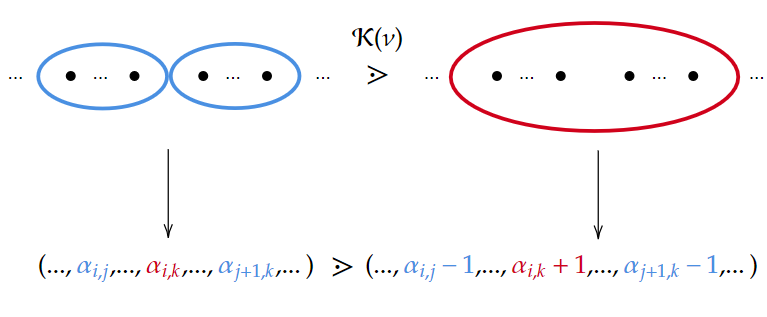}
    \caption{Merge order induced on Lusztig data}
    \label{fig:corluszdat}
\end{figure}
\begin{figure}[h]
    \centering
    \includegraphics[height = 0.175\textheight]{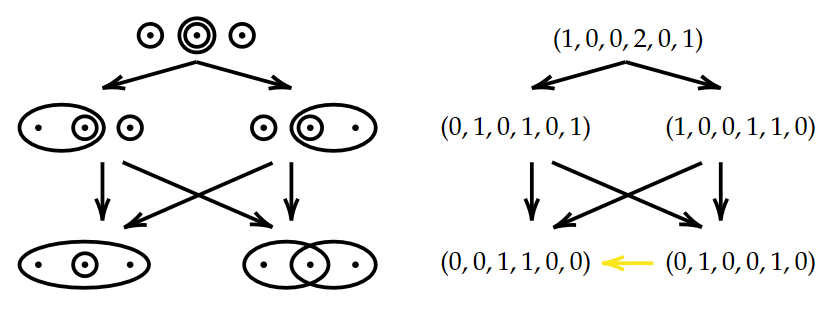}
    \caption{Hasse diagrams for $\varnu = (1,1,-1,-1)$}
    \label{fig:counterexample}
\end{figure}
\end{proof}
\begin{corollary}
The partial order on $A\left(\mathbf{h},\,-\sum{h_k}\right)$ refines the partial order on $\cT\left(\bh\right)$.
\end{corollary}
Naturally we would like to upgrade the merging order so that it will be equivalent to the two-sided dictionary order. Although an intuitive new order hasn't been made, there were some ideas below.
\begin{itemize}[label = $\bullet$]
    \item The \new{excess-merge} order works on cases similar to \Cref{fig:counterexample} by allowing any two intersecting loops cover two replaced loops --- one is the intersection and the other is the union. However, there exists a counterexample, for example $\bh = (1,1,-1,-1)$ (\Cref{fig:exmergecounter}).
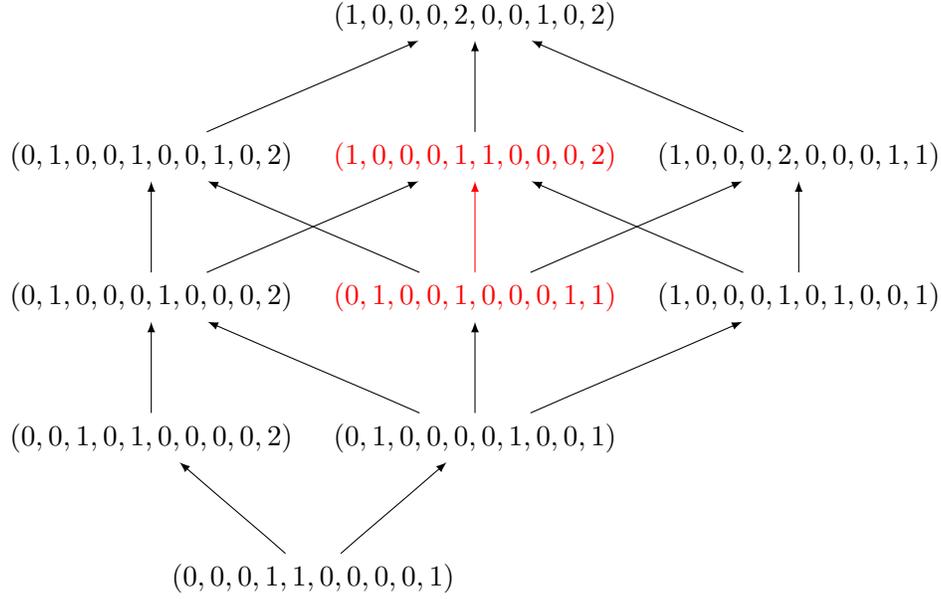
\begin{figure}[h]
    \centering
\begin{tikzpicture}[>=latex,line join=bevel,]
\node (node_0) at (113.0bp,8.5bp) [draw,draw=none] {$\left(0, 0, 0, 1, 1, 0, 0, 0, 0, 1\right)$};
  \node (node_1) at (52.0bp,61.5bp) [draw,draw=none] {$\left(0, 0, 1, 0, 1, 0, 0, 0, 0, 2\right)$};
  \node (node_2) at (174.0bp,61.5bp) [draw,draw=none] {$\left(0, 1, 0, 0, 0, 0, 1, 0, 0, 1\right)$};
  \node (node_3) at (52.0bp,114.5bp) [draw,draw=none] {$\left(0, 1, 0, 0, 0, 1, 0, 0, 0, 2\right)$};
  \node (node_4) at (174.0bp,114.5bp) [draw,draw=none] {$\textcolor{red}{\left(0, 1, 0, 0, 1, 0, 0, 0, 1, 1\right)}$};
  \node (node_6) at (296.0bp,114.5bp) [draw,draw=none] {$\left(1, 0, 0, 0, 1, 0, 1, 0, 0, 1\right)$};
  \node (node_5) at (52.0bp,167.5bp) [draw,draw=none] {$\left(0, 1, 0, 0, 1, 0, 0, 1, 0, 2\right)$};
  \node (node_7) at (174.0bp,167.5bp) [draw,draw=none] {$\textcolor{red}{\left(1, 0, 0, 0, 1, 1, 0, 0, 0, 2\right)}$};
  \node (node_8) at (296.0bp,167.5bp) [draw,draw=none] {$\left(1, 0, 0, 0, 2, 0, 0, 0, 1, 1\right)$};
  \node (node_9) at (174.0bp,220.5bp) [draw,draw=none] {$\left(1, 0, 0, 0, 2, 0, 0, 1, 0, 2\right)$};
  \draw [black,->] (node_0) ..controls (94.809bp,24.709bp) and (80.32bp,36.823bp)  .. (node_1);
  \draw [black,->] (node_0) ..controls (131.19bp,24.709bp) and (145.68bp,36.823bp)  .. (node_2);
  \draw [black,->] (node_1) ..controls (52.0bp,76.805bp) and (52.0bp,87.034bp)  .. (node_3);
  \draw [black,->] (node_2) ..controls (135.82bp,78.462bp) and (102.86bp,92.241bp)  .. (node_3);
  \draw [black,->] (node_2) ..controls (174.0bp,76.805bp) and (174.0bp,87.034bp)  .. (node_4);
  \draw [black,->] (node_2) ..controls (212.18bp,78.462bp) and (245.14bp,92.241bp)  .. (node_6);
  \draw [black,->] (node_3) ..controls (52.0bp,129.81bp) and (52.0bp,140.03bp)  .. (node_5);
  \draw [black,->] (node_3) ..controls (90.184bp,131.46bp) and (123.14bp,145.24bp)  .. (node_7);
  \draw [black,->] (node_4) ..controls (135.82bp,131.46bp) and (102.86bp,145.24bp)  .. (node_5);
  \draw [red,->] (node_4) ..controls (174.0bp,129.81bp) and (174.0bp,140.03bp)  .. (node_7);
  \draw [black,->] (node_4) ..controls (212.18bp,131.46bp) and (245.14bp,145.24bp)  .. (node_8);
  \draw [black,->] (node_5) ..controls (90.184bp,184.46bp) and (123.14bp,198.24bp)  .. (node_9);
  \draw [black,->] (node_6) ..controls (257.82bp,131.46bp) and (224.86bp,145.24bp)  .. (node_7);
  \draw [black,->] (node_6) ..controls (296.0bp,129.81bp) and (296.0bp,140.03bp)  .. (node_8);
  \draw [black,->] (node_7) ..controls (174.0bp,182.81bp) and (174.0bp,193.03bp)  .. (node_9);
  \draw [black,->] (node_8) ..controls (257.82bp,184.46bp) and (224.86bp,198.24bp)  .. (node_9);
    \end{tikzpicture}
        \caption{Excess-merge order counterexample with covering in red}
        \label{fig:exmergecounter}
    \end{figure}
    \item Taking into account \Cref{fig:exmergecounter}, the \new{excess-merge-shift} order also allows three (or more) adjacent loops shifting to form two (or more) adjacent loops. In this case, there also exists a counterexample, $\bh = (2,1,2,1)$ .
\end{itemize}

Additionally, it should be noted that, unlike the merge orders, the two-sided dictionary order is global. In other words, each covering relation depends on other values (in this case --- the outermost elements), which are not changed during the covering.

In the representation theory of canonical bases (which are in particular crystals and therefore indexed by Lusztig data) the two-sided dictionary order is used to show that Lusztig's canonical basis is in fact a basis. This is done by showing that the change of basis from a standard (PBW) basis $\{F^a\}$ to Lusztig's canonical basis $\{\Bar{F}^a\}$ is unitriangular with respect to the two-sided dictionary order (\Cref{thm:changeofbasis}).
\begin{theorem}[\cite{Tingley}]\label{thm:changeofbasis}
For every Lusztig data $a$,
\begin{align}
    \Bar{F}^{a} = F^{a} + \sum_{a' \prec a}^{} p^{a}_{a'}(q) F^{a'}, \nonumber
\end{align}
where the $p^{a}_{a'}(q)$ are Laurent polynomials in $q$.
\end{theorem}
The orders we consider lead us to ask if the change of basis holds true for any coarsening of the two-sided dictionary order.

\section{Probabilistic Tesler Matrices}
\label{s:ProbabilisticTesler}

\subsection{Introduction to Markov chains}
\label{s:IntroMarkovChain}

Consider a system of $n$ states and probabilities $t_{ij}$ (going from state $i$ to $j$) connecting them. The sum of all probabilities going from a state is equal to $1$. A step consists of going from one state to another randomly. This system is otherwise known as a \new{Markov chain}.

\begin{definition}
    The \new{Markov property} states that the probability for the next state after the $s+1$ step is solely based on its position after $s$ steps.
\end{definition}

A Markov chain can be represented as a graph and matrix (\Cref{fig:trmatrix}).

\begin{figure}[h]
    \centering
    \includegraphics[height = 0.2\textheight]{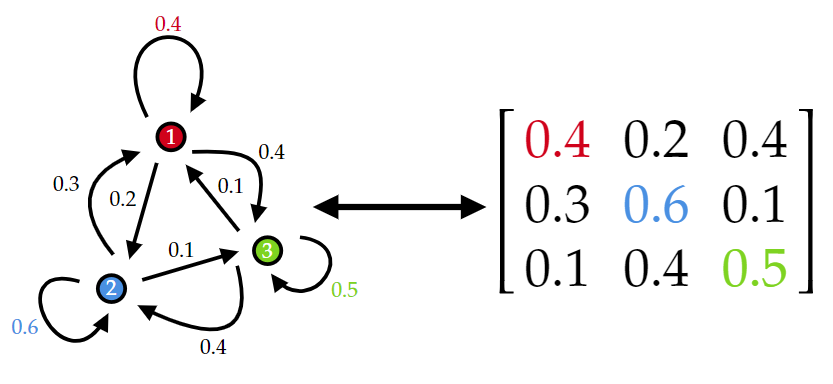}
    \caption{Graph and matrix representations of a Markov chain}
    \label{fig:trmatrix}
\end{figure}

\begin{definition}
    Matrix $T$, otherwise known as the \new{transition matrix}, has elements $t_{ij}$, which are the same as the probabilities, that are located in row $i$ and column $j$.
\end{definition}

\begin{definition}
    Our \new{state vector}, $V_{s}$, is the distribution between the states after $s$ steps. When $s = 0$, the state vector is \new{initial}. Each state vector follows the condition$$V_s = V_0 \cdot T^{s}.$$
\end{definition}


\begin{definition}
    In some cases, when $s \xrightarrow{} \infty$, all of the elements in $T^{s}$ converge. We call this the \new{equilibrium matrix} $E$. As a corollary,$$E=ET.$$
\end{definition}

\begin{example}
    Consider \Cref{fig:trmatrix}. The equilibrium matrix $E$ is equal to
    
\centering
$\begin{bmatrix}
\frac{4}{15}&\frac{13}{30}&\frac{3}{10}\\
\frac{4}{15}&\frac{13}{30}&\frac{3}{10}\\
\frac{4}{15}&\frac{13}{30}&\frac{3}{10}\\
\end{bmatrix}$   
\end{example}

\subsection{Markov chain on the Tesler poset}
\label{s:ProbabilityAlgorithms}

To establish probabilities on the elements, we need to define the probabilities of doing a step. Also, to effectively go through the entire poset, it is best to start at the maximal element (since there is only one) and do a covering only downwards, if possible.

The algorithm (\Cref{apx:pseudocode}) $\mathbf{inputs}$ the size of the matrix $s$, the number of random steps $r$, the number of overall trials $t$, and the hook sum $\mathbf{h}$. Each random step allows the matrix to go down a covering or stay. The $\mathbf{output}$ consists of: the \new{probability distribution over all ranks}, the \new{average ranks}, the \new{rank ratios} $\left(\frac{\text{rank}-1}{\text{range}}\right)$, and the \new{average step efficiency} (a step that goes down is effective and we count until a minimum element is reached) after $s$ steps.

Each step in the algorithm first consists of taking a random element in the upper-right part of the matrix. If it is on the main diagonal, then another element is chosen on the same diagonal and a downward covering is checked. Otherwise, we record the coordinates of the initial element $\left(a_x,\,a_y\right)$ and then choose a number from the range $1 \leq \dots < a_y < \dots < a_x < \dots \leq s$. Those 3 numbers are the same as $i,j,k$ in definition \Cref{def:PosetTesler}, from which we can decide whether a downward covering will work.

Let us consider the matrices $\left(1,\,1,\,1\right),\,\left(1,\,2,\,1\right),\,\left(3,\,1,\,2\right)$ as examples. Overall, the rank distribution should go to $1$ when $r$, the number of random steps, goes to infinity. \Cref{fig:rankdist} shows us how this would look like when $r$ is increased linearly.
Since the distribution gets shifted to $1$, when $r\to\infty$, the average step efficiency reaches a certain limit. For the earlier used matrices, we get $0.404788$, $0.359631$, and $0.428762$, respectively. 
Lastly, different matrices can be compared by looking at their rank ratio graphs. Each of these graphs start at $1$, when $r=0$, and a larger poset corresponds to a ``higher'' graph. This can all be seen on \Cref{fig:rankratio}. Also, each of these graphs grow similar to logarithmic graph.

%

\begin{figure}[h]

\centering

\begin{minipage}[b]{.5\textwidth}
\centering
\includegraphics[width=.9\textwidth]{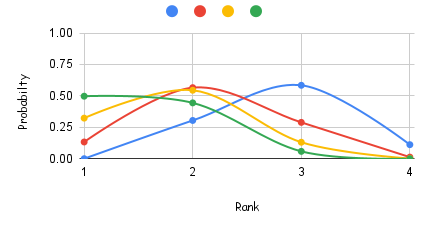}
\caption*{$\textcolor{blue}{2}$, $\textcolor{red}{4}$, $\textcolor{yellow}{6}$, $\textcolor{green}{8}$ random steps}
\end{minipage}%
\begin{minipage}[b]{.5\textwidth}
\centering
\includegraphics[width=.9\textwidth]{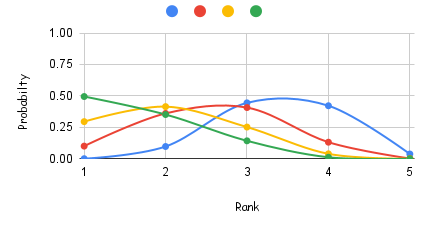}
\caption*{$\textcolor{blue}{3}$, $\textcolor{red}{6}$, $\textcolor{yellow}{9}$, $\textcolor{green}{12}$ random steps}
\end{minipage}

\bigskip

\begin{minipage}{.5\textwidth}
\centering
\includegraphics[width=.9\textwidth]{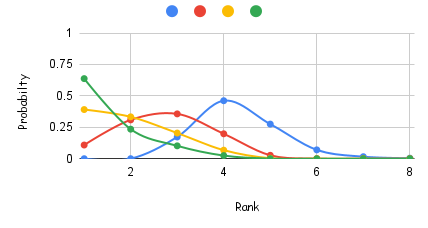}
\caption*{$\textcolor{blue}{5}$, $\textcolor{red}{10}$, $\textcolor{yellow}{15}$, $\textcolor{green}{20}$ random steps}
\end{minipage}

\caption{Rank distributions of matrices ($1,1,1$), ($1,2,1$), and ($3,1,2$), from top to right bottom}
\label{fig:rankdist}

\end{figure}

\begin{figure}[h]
    \centering
    \includegraphics[height = 0.3\textheight]{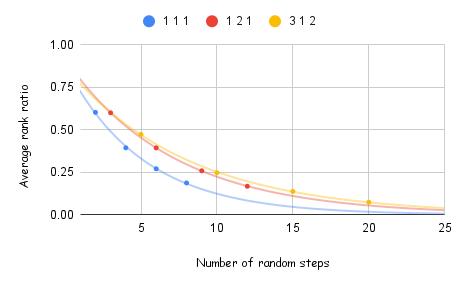}
    \caption{Rank ratio graphs with their logarithmic trends}
    \label{fig:rankratio}
\end{figure}


\section*{Future questions}
Finally, we shall formulate some questions that lack exploration in the paper. 

An obvious one is the problem of finding $\gamma_+$. It is probable that there is a cunning trick to find it at least for some height, expanding for every other. What about the case when $w\asymp\sqrt{h}$? We know $L\left[h,\,w\right]\asymp h\left(\asymp w\sqrt{h}\right)$, but is there a constant $\gamma_0$ as we have for other cases? Does the answer lie in the unification of $\gamma_+w\sqrt{h}$ and $w^2\ln\left(\frac{\sqrt{h}}{w}\right)$? And if we can complete the $\left[h,\,w\right]$ asymptotics, can we generalize \Cref{CorSum}?

Another question that arrises two-sided dictionary order is whether can be intuitively changed on the Kostant pictures in order to be equivalent to the Tesler poset. Or, whether the two-sided dictionary order can be reduced while still showing that Lusztig's canonical basis is a basis.

From the Markov chain on the Tesler matrices, an important question is how can we easily predetermine this data, such as the average step efficiency limit and the rank ratio graph? In other words, given the Tesler poset for a fixed $\mathbf{h}$, is there is a simple way to figure out these values?

\begin{figure}[h]
    \centering
    \includegraphics[height = 0.15\textheight]{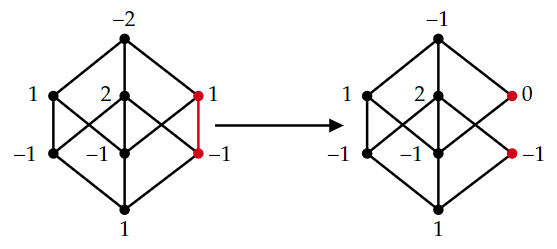}
    \caption{$\max_1\left|\mu\left(\hat{0},\,A\right)\right| = \max_2\left|\mu\left(\hat{0},\,A\right)\right| = 2$}
    \label{fig:mobiusalgorithm}
\end{figure}

It has been proposed \cite{ONeill} to use the Möbius function (its maximum absolute value) in order to analyze cardinality of $\T\left(\mathbf{1}^n\right)$. A question one may ask is whether this characteristic can be computed efficiently without generating the entire poset. Specifically, one can see that some edges, when removed, barely influence it (these would be more ``isolated'', \Cref{fig:mobiusalgorithm}). Is it possible to create an algorithm that would generate only an important part of a poset, leaving the ``isolated'' edges out of sight?



\section*{Acknowledgments}
\label{s:Acknowledgments}

We would like to thank Anne Dranowski for mentoring us, her guidance and support. We also thank Pavel Etingof, Slava Gerovitch, Dmytro Matvieievskyi, and the Yulia’s Dream program for providing this invaluable opportunity for us to do research. We are grateful to Nathan Williams, Alejandro Morales, and GaYee Park for their helpful contributions and suggestions to our work.
\section*{Appendix A: Asymptotics}
\label{apx:Asymptotics}
\begin{definition}
    Let $f,\,g$ be two integer/real-valued positive functions. As $x\to\infty$, we write
    \begin{enumerate}[label=(\roman*)]
        \item $f\sim g$ iff $f/g\to 1$;
        \item $f\prec g$ ($g\succ f$) iff $f/g\to0$;
        \item $f\asymp g$ iff $ag\leq f\leq bg$ for some $a,\, b\in\R^+$;
        \item $f\precsim g$ ($g\succsim f$) iff $f\leq ag$ for every $a>1$;
        \item $f\preccurlyeq g$ ($g\succcurlyeq f$) iff $f\leq ag$ for some $a\in\R^+$.
    \end{enumerate}
\end{definition}

\newpage
\section*{Appendix B: Pseudocode}
\label{apx:pseudocode}

\RestyleAlgo{ruled}

\SetKwComment{Comment}{/* }{ */}

\begin{algorithm}[ht!]
\label{alg:onewayTesMat}
\small

\caption{\small One way random Tesler matrix}

\LinesNumbered

\SetAlgoVlined

\KwData{Size of the matrix $s$, number of random steps $r$, the hook sum vector $\mathbf{h}$.}
\KwResult{Rank distribution, average rank ratio and average step efficiency.}
$s \gets $ size of the matrix;

$r \gets $ number of random steps;

$t \gets $ number of trials\Comment*[r]{$t$ is large in order to have accurate data}

\For{$1 \leq i \leq s$}{

$h_i \gets $ corresponding hook value input;

}
Calculate $\text{hook\_sum}$ and $\text{max\_sum}$ (sum of the elements in the maximal matrix);

Initialize $\text{rank\_array}[i]$ that holds the number of times the final element was at each rank;

\For{$t$ trials}{

\For {$r$ random steps}{
        
$\text{current\_sum}\gets$ the sum of all the elements in the matrix;

\If{$\text{current\_sum}=\text{hook\_sum}$}{
    The current trial reached the bottom of the poset;
}

$a_x \gets  x$ coordinate of randomly chosen element;

$a_y \gets y$ coordinate of randomly chosen element\Comment*[r]{$a_x \geq a_y$}

\If{$a_x = a_y$}{

$b_x \gets x$  coordinate of a different randomly chosen element on the diagonal;

$b_y \gets y$  coordinate of a different randomly chosen element on the diagonal;

\If{covering does not work}{

\If{we have not reached the end}{

$\text{unsuccessful\_attempts}\pp;$

}
Go back and do the next random step;

}

Do the covering and do the next random step;
}

Randomly choose a number from $1 \leq \dots < a_y < \dots < a_x < \dots \leq s;$

Rename coordinates $a_x,\,a_y$ and the number as $i,\,j,\,k;$

\If{covering does not work}{

\If{we have not reached the end}{

$\text{unsuccessful\_attempts}\pp;$
}
Go back and do the next random step;
}
Do the covering and do the next random step;
}

$\text{final\_sum} \gets $ the sum of all the elements in the matrix after $r$ randoms steps;

$\text{useful\_steps}=\text{max\_sum} - \text{final\_sum};$

 $\text{step\_efficiency} =\frac{\text{useful\_steps}}{\text{useful\_steps} + \text{unsuccessful\_attempts}};$
 
$\text{current\_rank} = \text{final\_sum} - \text{hook\_sum} + 1;$

$\text{rank\_array}[\text{current\_rank}]\pp;$

}
Calculate the average step efficiency over $t$ trials;

\For{1 $\leq i\leq$ $\left(\text{max\_sum} - \text{hook\_sum} + 1\right)$}{

rank $i$ probability $\gets\frac{\text{rank\_array}[i]}{t};$

}

Calculate $\text{average\_rank\_value};$

average rank ratio $\gets \frac{\text{average rank value} - 1}{\text{max\_sum} - \text{hook\_sum}};$

\end{algorithm}


\bibliographystyle{alpha}
\bibliography{tesler.bib}

\begin{thebibliography}{{Dre}14}

\bibitem[AK04]{AnKo}
Jared Anderson and Mikhail Kogan.
\newblock Mirkovi\'{c}-{V}ilonen cycles and polytopes in {T}ype {A}.
\newblock {\em Int. Math. Res. Not.}, pages 561--591, 2004.

\bibitem[CT15]{claxton2015youngtableauxmultisegmentspbw}
John Claxton and Peter Tingley.
\newblock Young tableaux, {multi-segments}, and {PBW} bases, 2015.

\bibitem[{Dre}14]{Armstrong2014}
{Drew Armstrong}.
\newblock Tesler matrices, 2014.
\newblock Talk slides at Bruce Saganfest.

\bibitem[Lus90]{lus90}
G.~Lusztig.
\newblock Canonical bases arising from quantized enveloping algebras.
\newblock {\em J. Amer. Math. Soc.}, 3(2):447--498, 1990.

\bibitem[MMR16]{M_sz_ros_2016}
Karola Mészáros, Alejandro~H. Morales, and Brendon Rhoades.
\newblock The polytope of tesler matrices.
\newblock {\em Selecta Mathematica}, 23(1):425–454, May 2016.

\bibitem[O'N18]{ONeill}
Jason O'Neill.
\newblock On the poset and asymptotics of {T}esler matrices.
\newblock {\em Electron. J. Combin.}, 25(2):Paper No. 2.4, 27, 2018.

\bibitem[Tin17]{Tingley}
Peter Tingley.
\newblock Elementary construction of {L}usztig's canonical basis.
\newblock In {\em Groups, rings, group rings, and {H}opf algebras}, volume 688 of {\em Contemp. Math.}, pages 265--277. Amer. Math. Soc., Providence, RI, 2017.

\bibitem[Wil24]{NathanWilliams2022}
Nathan Williams.
\newblock {\em Open Problems in Algebraic Combinatorics}, volume 110 of {\em Proceedings of Symposia in Pure Mathematics}.
\newblock American Mathematical Society, Providence, RI, \copyright 2024.

\bibitem[Zei99]{zeilberger1998proof}
Doron Zeilberger.
\newblock Proof of a conjecture of {C}han, {R}obbins, and {Y}uen.
\newblock volume~9, pages 147--148. 1999.
\newblock Orthogonal polynomials: numerical and symbolic algorithms (Legan\'es, 1998).

\end{thebibliography}

\end{document}